\documentclass[10pt,reqno]{amsart}
\usepackage{ucs}

\usepackage{amsmath, amssymb, amscd}
\usepackage{pb-diagram}

\usepackage[latin1]{inputenc}
\usepackage{amsmath}
\usepackage{amssymb}
\usepackage{epsfig}
\usepackage{graphicx}
\usepackage{psfrag}
\usepackage{xypic}

\newtheorem{theorem}{Theorem}[section]
\newtheorem{definition}[theorem]{Definition}

\newtheorem{proposition}[theorem]{Proposition}
\newtheorem{corollary}[theorem]{Corollary}

\newtheorem*{remark}{Remark}

\renewcommand{\epsilon}{\varepsilon}

\DeclareMathAlphabet{\mathpzc}{OT1}{pzc}{m}{it}
\usepackage{mathrsfs}

\newcommand{\Z}{\mathbb{Z}}
\newcommand{\C}{\mathbb{C}}

\renewcommand{\qed}{$\hfill \square$ \smallskip \\}
\renewcommand{\phi}{\varphi}

\begin{document}
\thispagestyle{empty}
\title[2-strand twisting \& knots with identical quantum knot homologies]{2-strand twisting and knots with identical quantum knot homologies}
\author{Andrew Lobb}

\email{lobb@math.sunysb.edu}
\address{Mathematics Department \\ Stony Brook University \\ Stony Brook NY 11794 \\ USA}

\begin{abstract}
Given a knot, we ask how its Khovanov and Khovanov-Rozansky homologies change under the operation of introducing twists in a pair of strands.  We obtain long exact sequences in homology and further algebraic structure which is then used to derive topological and computational results.  Two of our applications include giving a new way to generate arbitrary numbers of knots with isomorphic homologies and finding an infinite number of mutant knot pairs with isomorphic reduced homologies.\end{abstract}

\maketitle

\section{Introduction and results}
In this paper we consider $sl(n)$ Khovanov-Rozansky homology (Khovanov homology appears as $n=2$) under the operation of adding twists in a pair of strands.  We observe stabilization of the homology as we add more twists and, looking a little deeper, reveal some further algebraic structure which we exploit for various structural and topological results.

In the remainder of this paper we shall assume that we have chosen a fixed $n \geq 2$ unless we make it clear otherwise.

First we describe some chain complexes of matrix factorizations, one such for each integer, which will be the building blocks of this paper.

\begin{definition}
\label{Tkdefn}
For $k \geq 0$, the complex $T_k$ is the $sl(n)$ Khovanov-Rozansky chain complex of direct sums of matrix factorizations corresponding to a diagram of $k$ full twists in two oppositely oriented strands, where the $2k$ crossings are positive (see Figures \ref{t1} and \ref{tk} for an explicit picture).  When $k < 0$ we take the $-2k$ crossings to be negative.
\end{definition}

\noindent It should be clear that there is an obvious way in which each of these complexes can be built from $T_1$ and $T_{-1}$ by tensor product.

\begin{proposition}
Up to homotopy equivalence $T_k \otimes T_l = T_{k+l}$, where the tensor product of complexes of matrix factorizations is taken by concatenating in the obvious way the corresponding tangle diagrams with $|2k|$ and $|2l|$ crossings.
\end{proposition}

\begin{proof}
For $k$ and $l$ of the same sign this is by definition, and for $k$ and $l$ of opposite sign it follows from the invariance up to homotopy equivalence of the Khovanov-Rozansky chain complex under Reidemeister move \emph{II}.
\end{proof}

\begin{figure}
\centerline{
{
\psfrag{pnodes}{$2p - 1$ nodes}
\psfrag{+}{$+$}
\psfrag{-}{$-$}
\psfrag{ldots}{$\ldots$}
\psfrag{T(D)}{$T(D)$}
\psfrag{T-(D)}{$T^-(D)$}
\psfrag{T+(D)}{$T^+(D)$}
\includegraphics[height=1in,width=1.5in]{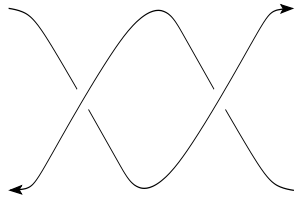}
}}
\caption{The complex $T_1$ is the $sl(n)$ Khovanov-Rozansky complex of direct sums of matrix factorizations corresponding to this diagram.  Note that there are two positive crossings in the diagram.}
\label{t1}
\end{figure}

\begin{figure}
\centerline{
{
\psfrag{pnodes}{$2p - 1$ nodes}
\psfrag{+}{$+$}
\psfrag{-}{$-$}
\psfrag{ldots}{$\ldots$}
\psfrag{T(D)}{$T(D)$}
\psfrag{T-(D)}{$T^-(D)$}
\psfrag{T+(D)}{$T^+(D)$}
\includegraphics[height=1in,width=5in]{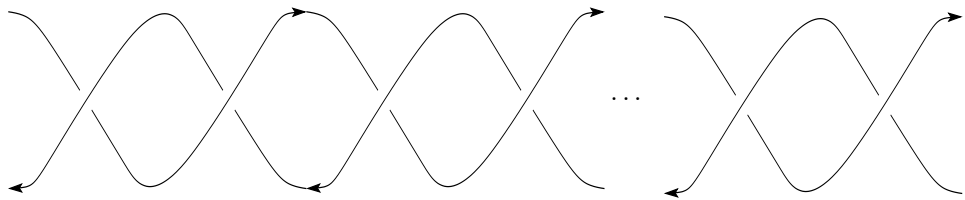}
}}
\caption{The complex $T_k = \otimes^k T_1$ is the $sl(n)$ Khovanov-Rozansky complex corresponding to the diagram above with $2k$ crossings.}
\label{tk}
\end{figure}

There are two main sections with proofs in this paper: in Section \ref{stabsection} we shall deal with the question of stabilization of the complex $T_k$ as $k \rightarrow \infty$ and prove results necessary for the topological and structural results proved in Section \ref{topsection}.  For the rest of the current section we give statements of some results whose proofs follow later and context for these results.

\subsection*{Acknowledgements}

This paper would not have been possible without the input of Daniel Krasner.  Akio Kawauchi has supplied invaluable advice concerning his paper \cite{Kaw}. The author would also like to thank Cameron Gordon, Marc Lackenby, Jacob Rasmussen, and Scott Taylor for useful emails or discussions.  Thanks to Josh Greene and Dylan Thurston for comments on earlier versions of this paper.

Much of the work in this paper was done during the Knot Homology semester at MSRI, where the author was funded as a postdoctoral fellow.

\subsection{Stabilization and exact sequences}

In this section, all complexes are understood to be complexes of matrix factorizations, and $C(K)$ and $H(K)$ stand for the $sl(n)$ Khovanov-Rozansky chain complex and homology of the knot $K$ respectively, for some fixed $n \geq 2$.  Sometimes we will mean specifically the reduced, unreduced, or equivariant (with potential $w=x^{n+1} - ax$) \cite{Kras2} homologies in which case we shall make it clear.  Otherwise results should be interpreted as holding for each of these three versions of Khovanov-Rozansky homology.

By \emph{stabilization} we mean, most basically, the existence of a complex $T_\infty$, the direct limit of a sequence of maps $T_k \rightarrow T_{k+1}$.  This complex $T_\infty$ is defined in Definition \ref{2defo}.

If we have a knot $K$ given by a diagram $D$ we may consider $T_0$ as a subtangle of $D$.  Replacing $T_0$ by $T_1, T_2, T_3, \ldots$ in $D$ we obtain a sequence of diagrams $D_1, D_2, D_3, \ldots$ and hence a sequence of knots $K_1, K_2, K_3, \ldots$.

In the chain complex $C(D_i)$, $T_i$ appears as a tensor factor.  Replacing $T_i$ by $T_\infty$ gives us a chain complex which we shall denote $C(D_\infty)$ and its homology by $H(D_\infty)$.  We have, in effect, replaced the $T_i$ tangle in $D_i$ by a ``tangle consisting of an infinite number of twists''.

In the following theorems we let $D$ be such a diagram with a subtangle of $D$ identified with $T_0$.  We write $c_-$ and $c_+$ for the number of negative crossings and for the number of positive crossings of $D$ respectively.

\begin{theorem}
\label{stab1}
For each $0 \leq i < j$ there exists a directed system of maps (to be defined)

\[ F_{i,j} : T_i \rightarrow T_j \]

\noindent that is graded of homological degree $0$ and of quantum degree $0$.  Then for $0 \leq i < j$ (we allow $j = \infty$) we have that the induced map on homology

\[ F_{i,j} : H(D_i) \rightarrow H(D_j) \]

\noindent is an isomorphism in all homological degrees $\leq 2i - c_- - 2$.

\end{theorem}

\noindent Using square brackets to denote a shift in homological grading, and curly brackets to denote a shift in quantum grading, we also have:

\begin{theorem}
\label{stab2}
For each $0 \leq i < j$ there exists a directed system of maps (to be defined)

\[ G_{i,j} : T_i \rightarrow T_j [ 2(i-j) ] \{ 2n(j-i) \} \]

\noindent that is graded of homological degree $0$ and of quantum degree $0$.  Then for $0 \leq i < j$ we have that the induced map on homology

\[ G_{i,j} : H(D_i) \rightarrow H(D_j) [ 2(i-j) ] \{ 2n(j-i) \}  \]

\noindent is an isomorphism in all homological degrees $\geq c_+$.
\end{theorem}

\begin{remark}
To shorten our exposition, in this paper we restrict ourselves to the tangles $T_k$ where the $2k$ crossings are positive.  For each theorem we state, there is a dual theorem using negative crossings that the interested reader should have no trouble in stating and proving for herself.
\end{remark}

If this were all that there were to say about the algebra, we would not expect to be able to prove interesting results.  However, the maps $F_{i,j}$ and $G_{i,j}$ mesh well together, in a sense that we shall later make explicit.

From homology theories in different branches of mathematics we know that short exact sequences of chain complexes (and hence long exact sequences of homology groups) are useful tools when they are found in a homology theory.  And even more so are morphisms of short exact sequences of chain complexes (giving natural maps between long exact sequences of homology groups).  We find these relatively easily in our set-up and it is these that provide the power to start proving our later topological and structural results.

The results on exact sequences are best stated in the next section, after Theorems \ref{stab1} and \ref{stab2} are established.  For those wishing to jump ahead, these results appear as Propositions \ref{usefulseq} and \ref{anotherseq}.

We do not expect that the topological and structural corollaries that we find represent all of that which can be proved by making use of our exact sequences.  We therefore end this subsection with an encouragement for others to play with these exact sequences and see what else may drop out!

\subsection{Topological and structural results on Khovanov-Rozansky homology}

In \cite{Ras1} Rasmussen gives a homomorphism $s : K \mapsto s(K) \in 2\Z$ from the smooth knot concordance group to the additive group of even integers.  Furthermore, he shows that $s$ provides a lower bound $| s(K) | / 2$ on the smooth slice genus of a knot $K$.  Rasmussen's construction proceeds by extracting an even integer $s(K)$ from the $E_\infty$ page of a spectral sequence which has $E_2$ page the standard Khovanov homology of $K$.  This spectral sequence is essentially due to Lee \cite{Lee}.

Since this seminal paper, there have been generalizations of this result for other quantum knot homologies.  In particular Gornik \cite{G} has constructed a spectral sequence with $E_2$ page $sl(n)$ Khovanov-Rozansky homology $H(K)$.  In \cite{L3}, the author shows that the $E_\infty$ page of Gornik's spectral sequence is equivalent to an even integer $s_n(K)$ which gives a homomorphism $s_n : K \mapsto s_n(K) \in 2\Z$ from the smooth knot concordance group to the additive group of even integers.  Earlier work by the author \cite{L1} and independently by Wu \cite{Wu1}, implies that $|s_n(K)|/2(n-1)$ is a lower bound on the smooth slice genus of $K$.

In \cite{L3} it is shown that the $E_\infty$ page of Gornik's spectral sequence is isomorphic as a graded group to the homology of the unknot but with a shift in quantum grading $E_\infty \cong  H(U) \{ s_n(K) \}$, so that all the information about $E_\infty$ is contained in the even integer $s_n(K)$.

In \cite{Ras1}, Rasmussen asked if the concordance homomorphism $s$ coming from Khovanov homology was the same as the concordance homomorphism $\tau$ coming Heegaard-Floer knot homology, a conjecture motivated by the observation that $s$ and $\tau$ share many of the same properties.  A negative answer to Rasmussen's question was first provided by Hedden and Ording \cite{HO}.

The homomorphisms $s_n$ also share many properties with $s$ and moreover both $s$ and $s_n$ arise from the quantum world.  It is an interesting open question whether the homomorphisms $s_n$ are equivalent to the homomorphism $s = -s_2$ (see Conjecture $1.5$ of \cite{L3}).  Partly as a first step towards this question, in this paper we give a way in which the standard Khovanov-Rozansky homology interacts with $s_n$.

\begin{theorem}
\label{changesn}
Let the knot $K_0$ be obtained by changing a crossing of $K_{-1}$ from negative to positive as in Figure \ref{crossdiag}.  Then we know by \cite{L3} and Corollary $3$ of Livingston's \cite{Liv} that we must have

\[ s_n(K_0) \leq s_n(K_{-1}) {\rm .} \]

\noindent If in fact we have strict inequality $s_n(K_0) < s_n(K_{-1})$ then the homology group in homological degree $2p$ satisfies

\[ H^{2p}(K_p) \not= 0 \]

\noindent for the sequence of knots $K_1, K_2, \ldots$ shown in Figure \ref{crossdiag}.
\end{theorem}

\begin{figure}
\centerline{
{
\psfrag{K0}{$D_0$}
\psfrag{K-1}{$D_{-1}$}
\psfrag{Ki}{$D_p$}
\psfrag{ldots}{$\ldots$}
\psfrag{T(D)}{$T(D)$}
\psfrag{T-(D)}{$T^-(D)$}
\psfrag{T+(D)}{$T^+(D)$}
\includegraphics[height=1.5in,width=5in]{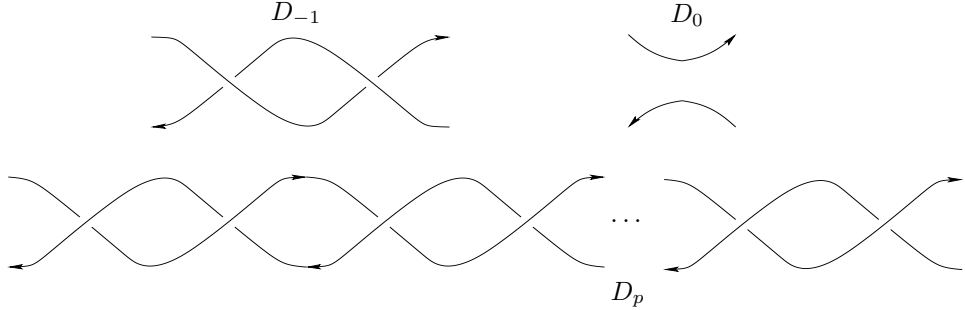}
}}
\caption{Here we show a knot $K_{-1}$ differing from a knot $K_0$ by a single crossing change.  We have drawn local pictures of diagrams of these knots.  The knots $K_p$ for $p \geq 1$ have diagrams $D_p$ formed by making $p$ further positive crossing changes at the same site as shown.  Alternatively, one can think of the knot $K_{-1}$ and the knots $K_p$ as obtained from $K_0$ by replacing the tangle $T_0$ shown in $D_0$ by $T_{-1}$ or $T_p$ respectively.}
\label{crossdiag}
\end{figure}

By the definition of the Khovanov-Rozansky chain complex it is then clear that we have the following:

\begin{corollary}
\label{cheapcor}
Given the conditions of Theorem \ref{changesn} the knot $K_p$ must have at least $2p$ positive crossings in any diagram. \qed
\end{corollary}

\noindent In other words, the crossings in $K_p$ shown in Figure \ref{crossdiag} are in some sense essential.  We note that each $s_n$ provides a tight bound on the the unknotting number of a torus knot and in the standard diagram of a torus knot, a single crossing change anywhere results in a diagram with a smaller unknotting number.  Hence Corollary \ref{cheapcor} can be applied in this situation.

The exact sequences that we are using work best when we can identify one of the terms.  In particular, we expect to be able to say useful things about knots with unknotting number equal to 1.

\begin{theorem}
\label{snunknot}
We consider the situation of Figure \ref{crossdiag} where we take $K_0 = U$, the unknot.  Then we have

\[ s_n(K_p) = s_n(K_1) \]

\noindent for all $p \geq 1$.
\end{theorem}

We mentioned above that sometimes by $H(K)$ we shall mean the equivariant Khovanov-Rozansky homology \cite{Kras2} with potential $w = x^{n+1} - ax$.  Here, all the modules involved in the Khovanov-Rozansky complex are free $\C[a]$-modules where $a$ has quantum grading $2n$.  The reason we are interested in this version of Khovanov-Rozansky homology is that the $s_n$ invariant is then built into the homology.  In fact for any knot $K$, we have that the equivariant homology with this potential satisfies

\[ H(K) = tor \oplus \bigoplus_{l=1}^{n}\C[a] [0] \{ 2l - n - 1 + s_n(K) \} \]

\noindent where $tor$ is a finitely-generated torsion $\C[a]$-module.

To see this, observe that $C(K)$, as a freely-generated graded complex of $\C[a]$-modules is chain homotopy equivalent to a sum of complexes of the form

\begin{enumerate}
\item $0 \rightarrow \C[a] \rightarrow 0$ and
\item $0 \rightarrow \C[a] \stackrel{a^k}{\rightarrow} \C[a] \rightarrow 0$.
\end{enumerate}

\noindent Setting $a = 0$ we recover standard Khovanov-Rozansky homology, while setting $a=1$ destroys the quantum grading and gives us Gornik's version of Khovanov-Rozansky homology.  This also tells us that nothing is lost by considering equivariant homology since the non-equivariant unreduced homology can be obtained from the equivariant homology groups.

In the case where $s_n(K_1) = 0$, we can say more about the homology of the knot $K_p$.  In fact, the homology of $K_p$ is characterized entirely by $p$ and the homology of the knot $K_1$.  We state this first for the equivariant case.

\begin{theorem}
\label{knotsfromU}
We consider the situation of Figure \ref{crossdiag} where we take $K_0 = U$, the unknot, and assume that $s_n(K_1) = 0$.  Taking equivariant homology with potential $w= x^{n+1} - ax$, let $\Delta$ be the bigraded $\C[a]$-module isomorphic to the torsion part of $H(K_1)$.  Then for $p \geq 2$ we have

\[ H(K_p) = H(K_{p-1}) \oplus \Delta[2p]\{ 2n(1-p) \} {\rm .}\]

\end{theorem}

It is almost possible to characterize completely the homology of $K_p$ in terms of $p$ and the homology of $K_{p-1}$ even if $s_n(K_1) \not=0$.  In fact, just knowing $H(K_{p-1})$ we would know $H(K_p)$ in all homological degrees apart from possibly one, and to determine $H(K_p)$ in this degree we would need one more piece of information.  We discuss what piece of information this is following the proof of Theorem \ref{knotsfromU}.  Armed with Theorem \ref{knotsfromU}, we can also consider the non-equivariant cases.

\begin{theorem}
\label{knotsfromU2}
Suppose we are in the set-up of Theorem \ref{knotsfromU} and let $H(K)$ stand for the standard unreduced or reduced Khovanov-Rozansky homology of $K$.  Let $\Delta$ be the bigraded $\C$-module satisfying

\[ H(K_1) = \C[0]\{0\} \oplus \Delta\]

\noindent for the reduced case and

\[ H(K_1) = \C[0]\{1-n\} \oplus \C[0]\{3-n\} \oplus \cdots \oplus \C[0]\{n-1\} \oplus \Delta \]

\noindent for the unreduced case.  Then for $p \geq 2$ we have

\[ H(K_p) = H(K_{p-1}) \oplus \Delta[2p]\{ 2n(1-p) \} {\rm .}\]
\end{theorem}

By relating Khovanov homology with their own instanton knot Floer homology, Kronheimer and Mrowka have shown that Khovanov homology detects the unknot \cite{KrMr}.  It is still an open question whether the Jones polynomial, which is the graded Euler characteristic of Khovanov homology, detects the unknot.  However, it is known that the Jones polynomial (and likewise the HOMFLY polynomial) does not enjoy the stronger property of being a complete invariant able to distinguish between any pair of knots.  For example, the HOMFLY polynomial is unable to distinguish between mutant knots.

It has been verified by Mackaay and Vaz \cite{MV} that the mutant knot pair consisting of the Kinoshita-Terasaka and the Conway knots have isomorphic reduced Khovanov-Rozansky homologies and hence also isomorphic reduced HOMFLY homologies.  Furthermore, there exist families of distinct 2-bridge knots with the same HOMFLY polynomials.  Since 2-bridge knots have thin homology, these knots must also share isomorphic reduced Khovanov-Rozansky homologies.

With Theorem \ref{knotsfromU} on hand we can give a new method for producing families of knots with isomorphic Khovanov-Rozansky homologies.  The next theorem follows as a consequence.

\begin{theorem}
\label{noncomplete}
Given a natural number $m$, there are $m$ distinct prime knots with bridge number greater than 2, which have isomorphic $sl(n)$ Khovanov-Rozansky homologies for all $n$.
\end{theorem}

We note that Theorem \ref{noncomplete} holds for reduced, unreduced, and equivariant homology with potential $w = x^{n+1} + ax$.  The knots undistinguished by these flavors of Khovanov-Rozansky homology that we produce are not necessarily thin nor necessarily related by mutation.  For an example of two knots with isomorphic Khovanov-Rozansky homologies, see Figure \ref{megaexample} and the discussion in Subsection \ref{examplesandthat}.

It remains a motivating question whether topological conclusions may be drawn from the coincidence of Khovanov-Rozansky homologies.  Further consequences of Theorem \ref{noncomplete} and its proof are discussed in Subsections \ref{examplesandthat} and \ref{mutantpairs}, where we give specific examples of interesting phenomena including an infinite number of mutant knot pairs with isomorphic reduced homologies.

\section{Algebraic structure results}
\label{stabsection}
In this section we shall prove Theorems \ref{stab1} and \ref{stab2} and derive further results enabling us to prove our more topological theorems.

\subsection{Stabilization}
To simplify notation we shall write $V$ and $Z$ (\emph{v}ertical and hori\emph{z}ontal) for the matrix factorizations indicated in Figure \ref{VandZ}.

\begin{figure}
\centerline{
{
\psfrag{V}{$V$}
\psfrag{Z}{$Z$}
\psfrag{x1}{$x_1$}
\psfrag{x2}{$x_2$}
\psfrag{x3}{$x_3$}
\psfrag{x4}{$x_4$}
\psfrag{T+(D)}{$T^+(D)$}
\includegraphics[height=1.5in,width=3.5in]{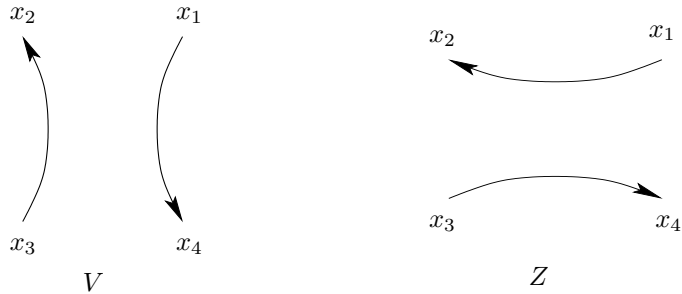}
}}
\caption{We draw here the matrix factorizations $V$ and $Z$.  In the text of this paper, $V$ and $Z$ often appear with integers appended in curly and/or square parentheses to indicate quantum degree shift and homological degree respectively.}
\label{VandZ}
\end{figure}

In \cite{Kras1}, Krasner gave a compact description of the complex $T_k$ of our Definition \ref{Tkdefn}.  As a consequence of this description, one can see that knot diagrams built up from these tangle building blocks have associated chain complexes which avoid the ``thick-edged'' matrix factorization, and hence much of the complication usually involved in the Khovanov-Rozansky chain complex.  Understanding such \emph{Krasner knots} may well be a good way to begin getting a grasp on Khovanov-Rozansky homology.

This compact description of $T_k$ will essentially be our main ingredient.  In the theorem that follows we use curly or square parentheses to indicate shift in the quantum degree and homological degree respectively and $w$ is the \emph{potential}.  We state Krasner's theorem both for the standard potential $w = x^{n+1}$ and for the equivariant potential $w = x^{n+1} - ax$, although Krasner only stated it for the standard potential.  Since the results that go into the proof of Krasner's theorem have now been established in the general equivariant setting \cite{Kras2}, we can state the result in more generality.

\begin{theorem}[Krasner \cite{Kras1}]
\label{input}
Up to chain homotopy equivalence, the complex $T_k$ is isomorphic to the following chain complex of matrix factorizations:

\begin{eqnarray*}
V[0] \{ 1-n \} &\stackrel{x_2 - x_4}{\longrightarrow}& V[1]\{-1-n\} \stackrel{A}{\longrightarrow} V[2]\{1-3n\}  \stackrel{x_2 - x_4}{\longrightarrow} \cdots \\
\cdots &\stackrel{x_2 - x_4}{\longrightarrow}& V[2k-1]\{(1-2k)n - 1 \}  \stackrel{S}{\longrightarrow}Z[2k]\{-2kn\} {\rm ,}
\end{eqnarray*}

\noindent where we write

\[ A =  x_2^{n-1} + x_2^{n-2}x_4 + x_2^{n-3}x_4^2 + \cdots + x_4^{n-1} \]

\noindent and we write $S$ for the map induced by the saddle cobordism.
\end{theorem}

\begin{definition}
\label{2defo}
Setting $k=\infty$ in Theorem \ref{input} gives us a definition of a complex $T_\infty$.
\end{definition}

With Krasner's characterization, it is a quick matter to define the chain maps $F_{i,j}$ and $G_{i,j}$ of Theorems \ref{stab1} and \ref{stab2}.

\begin{definition}
Let $0 \leq i < j$.  Using the description of Theorem \ref{input} of the complexes $T_k$, we define two maps 

\[ F_{i,j} : T_i \rightarrow T_j {\rm ,}\]

\[ G_{i,j} : T_i \rightarrow T_j [2(i-j)]\{2n(j-i)\} \]

\noindent as follows.  We require that $F_{i,j}$ preserves the homological grading and is the identity map on the the matrix factorizations in all homological degrees less than $2i$.  To the component of $F_{i,j}$ in homological degree $2i$ we assign the map $S' = (-1/n+1)S$ where $S$ is the map of matrix factorizations associated to the saddle cobordism.  To check that $F_{i,j}$ is a chain map, it is enough to observe that

\[ S^2 = -(n+1)A \, \, {\rm and} \, \, (x_2 - x_4)\circ S = 0{\rm .}\]

\noindent The former of these identities is computed in detail in Appendix A of \cite{KRS2}.  For the latter note that up to homotopy we have

\begin{eqnarray*}
(x_2 - x_4)\circ S &=& x_2 \circ S - x_4\circ S \\
&=& x_2 \circ S - S \circ x_4 = x_2 \circ S - S \circ x_1 \\
&=&  x_2 \circ S - x_1 \circ S = x_2 \circ S - x_2 \circ S \\
&=& 0 \rm{.}
\end{eqnarray*}

\noindent Clearly $F_{i,j}$ preserves the quantum grading.

We require that $G_{i,j}$ is the identity map on all homological degrees of $T_i$ which are non-zero matrix factorizations.  Certainly then $G_{i,j}$ is a chain map and we see that it is quantum graded of degree $0$.
\end{definition}

With these definitions in hand, the path to proving Theorems \ref{stab1} and \ref{stab2} is straightforward: in brief, we compute the cones of the maps $F_{i,j}$ and $G_{i,j}$ and show that the homology of the cones is supported well away from certain degrees in which $F_{i,j}$ and $G_{i,j}$ must therefore induce isomorphisms.

In the following propositions, we leave out quantum grading shifts and only give the leftmost and rightmost homological gradings.  We do this in order to try and give an uncluttered exposition; for the reader who is making use of these propositions, we recommend having a copy of Krasner's \cite{Kras1} to hand.

\begin{proposition}
\label{coneF}
Writing $Co(F_{i,j})$ for the cone of $F_{i,j}$ we have

\[ Co(F_{i,j}) = Z[2i-1] \stackrel{S'}{\rightarrow} V \stackrel{x_2 - x_4}{\rightarrow} V \stackrel{A}{\rightarrow} V \stackrel{x_2 - x_4}{\rightarrow} V  \stackrel{A}{\rightarrow} V \cdots \stackrel{x_2 - x_4}{\rightarrow} V  \stackrel{S}{\longrightarrow} Z[2j] {\rm .} \]
\end{proposition}

\begin{proposition}
\label{coneG}
Writing $Co(G_{i,j})$ for the cone of $G_{i,j}$ we have

\[ Co(G_{i,j}) = V[2(i-j)]  \stackrel{x_2 - x_4}{\rightarrow} V \stackrel{A}{\rightarrow} V \stackrel{x_2 - x_4}{\rightarrow} V  \stackrel{A}{\rightarrow} V \cdots \stackrel{x_2 - x_4}{\rightarrow} V[-1] {\rm .} \]
\end{proposition}

\begin{proof}[Proof of Propositions \ref{coneF} and \ref{coneG}]
This is a straightforward application of Gaussian elimination.  Starting from the leftmost homological degree in the case of $F_{i,j}$ and the rightmost in the case of $G_{i,j}$, we cancel all the identity maps of chain factorizations appearing as components of the chain maps.
\end{proof}

With our precise knowledge of the cones $Co$ of the chain maps $F_{i,j}$ and $G_{i,j}$, it is straightforward to prove our stabilization Theorems \ref{stab1} and \ref{stab2}.

\begin{proof}[Proof of Theorems \ref{stab1} and \ref{stab2}]
There is a short exact of chain complexes

\[ 0 \rightarrow C(D_i) \stackrel{F_{i,j}}{\rightarrow} C(D_j) \rightarrow Co(F_{i,j}) \rightarrow 0 \]

\noindent in which each map is graded of homological and quantum degree $0$.  This is clear in the unreduced and equivariant settings, and indeed holds also in the reduced setting since the map of rings $\C[x]/x^n \rightarrow \C$ is flat.

Induced by this short exact sequence is a long exact sequence of homology groups.  Proposition \ref{coneF} tells us that we must have

\[ H^k (Co(F_{i,j})) = 0 \]

\noindent for $k \leq 2i - c_- - 2$, so that the long exact sequence consists of isomorphisms $F_{i,j}$ in homological degrees $\leq 2i - c_- - 2$.  This proves Proposition \ref{stab1}.

The proof of Theorem \ref{stab2} follows the same argument.
\end{proof}

Informally speaking, Theorem \ref{stab1} tells us that we can generalize the class of objects for which there exists Khovanov-Rozansky homology to include \emph{knots with infinite twist regions}, as discussed in the preamble to the statement of the theorem.  More formally we could consider knot diagrams with extra singularities allowed.  The concept is outlined in Figure \ref{inftwist}.  Investigating these stable homologies is an interesting project, but we shall not pursue it further in this paper.

\begin{figure}
\centerline{
{
\psfrag{+inf}{$+\infty$}
\psfrag{Z}{$Z$}
\psfrag{x1}{$x_1$}
\psfrag{x2}{$x_2$}
\psfrag{x3}{$x_3$}
\psfrag{x4}{$x_4$}
\psfrag{T+(D)}{$T^+(D)$}
\includegraphics[height=1.5in,width=2in]{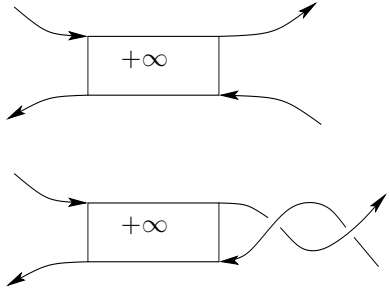}
}}
\caption{We show an example of part of a knot diagram where we have allowed an extra type of singularity corresponding to an infinitely positively twisted pair of strands.  From the results on stabilization in this paper it follows that such enhanced diagrams have well-defined homology groups.  We also give an example of a new Reidemeister-type move for such diagrams: the infinitely twisted region donates a positive twist to the rest of the diagram.  Clearly the homology groups will not change under this move.  Further moves are possible of course, and we encourage the reader to investigate.}
\label{inftwist}
\end{figure}

\subsection{Some exact sequences}

We derive some exact sequences of homology groups to use in proving the structural and topological theorems of Section \ref{topsection}.

\begin{proposition}
\label{usefulseq}
Let knots $K_0, K_1, K_2, \ldots$ be given as in Figure \ref{crossdiag}.  Then there is a commutative diagram in which the rows are exact, which has the following form:

\[\minCDarrowwidth15pt\begin{CD}
@>>> M^{-1} @>>> H^0(K_0) @>>> H^0(K_1) @>>> M^0 @>>>H^1(K_0) @>>> \\
 @. @V id VV @VVV @VVV @V id VV @VVV\\
 @>>> M^{-1} @>>> H^2(K_1)\{2n\} @>>> H^2(K_2)\{2n\} @>>> M^0 @>>>H^3(K_1)\{2n\} @>>> \\
 @. @V id VV @VVV @VVV @V id VV @VVV\\
@>>> M^{-1} @>>> H^4(K_2)\{4n\} @>>> H^4(K_3)\{4n\} @>>> M^0 @>>>H^5(K_2)\{4n\} @>>> \\
 @. @V id VV @VVV @VVV @V id VV @VVV
\end{CD}\]

\noindent Here $M$ is a bigraded, finitely-generated module (over $\C$ or $\C[a]$ depending on the variant of homology chosen).  Moreover, in the equivariant case, $M$ is a torsion $\C[a]$-module.  All maps preserve the quantum grading.
\end{proposition}

\begin{proof}
Each row of the commutative diagram comes about from a short exact sequence of chain complexes, and the maps between the rows are induced by morphisms of these short exact sequences.  From Proposition \ref{coneF} we first observe that

\[ Co(F_{i, i+1}) = Co(F_{0,1})[2i]\{-2ni\} {\rm .} \]

\noindent It is then straightforward to check that for $i \geq 0$, there is a commutative map of short exact sequences of chain complexes

\[\minCDarrowwidth20pt\begin{CD}
0 @>>> C(D_i) @>F_{i, i+1}>> C(D_{i+1}) @>>> Co(F_{0,1})[2i]\{-2ni\} @>>> 0\\
@. @VG_{i,i+1}VV @VG_{i+1, i+2}VV @VidVV @.\\
0 @>>> C(D_{i+1})[-2]\{2n\} @>F_{i+1, i+2}>> C(D_{i+2})[-2]\{2n\} @>>> Co(F_{0,1})[2i]\{-2ni\} @>>> 0{\rm .}
\end{CD}\]

So setting $M = H(Co(F_{0,1}: C(D_0) \rightarrow C(D_1)))$ we are almost done, it only remains to argue that $M$ is finitely-generated and, in the equivariant case, torsion.

That $M$ is finitely-generated follows from $H(K_0)$ and $H(K_1)$ being finitely-generated and the first row of the commutative diagram.  In the equivariant case, suppose that $M$ were not torsion, so that there is some $i$ for which $\C[a]$ is a submodule of $M^i$.  Taking a row low enough in the commutative diagram, we see that this would force $H^k(K_l)$ to be non-torsion for some $k > 0$ and some knot $K_l$, a contradiction. Hence $M$ is torsion.
\end{proof}

\begin{proposition}
\label{anotherseq}
Let knots $K_0, K_1, K_2, \ldots$ be given as in Figure \ref{crossdiag}.  Then there is a commutative diagram in which the rows are exact, which has the following form:

\[\minCDarrowwidth17pt\begin{CD}
@>>> N^{-1} @>>> H^0(K_0) @>>> H^2(K_1)\{2n\} @>>> N^0 @>>> H^1(K_0) @>>>\\
@. @VidVV @VVV @VVV @VidVV @VVV \\
@>>> N^{-1} @>>> H^0(K_1) @>>> H^2(K_2)\{2n\} @>>> N^0 @>>> H^1(K_1) @>>> \\
@. @VidVV @VVV @VVV @VidVV @VVV \\
@>>> N^{-1} @>>> H^0(K_2) @>>> H^2(K_3)\{2n\} @>>> N^0 @>>> H^1(K_2) @>>>\\
@. @VidVV @VVV @VVV @VidVV @VVV
\end{CD}\]

Here $N$ is a bigraded, finitely-generated module (over $\C$ or $\C[a]$ depending on the variant of homology chosen).  Every map in the complex preserves the quantum grading.
\end{proposition}

\begin{proof}
Setting $N = H(Co(G_{0,1}: C(D_0) \rightarrow C(D_1)))$, this follows in the same way as before from the commutative map of short exact sequences:

\[\minCDarrowwidth20pt\begin{CD}
0 @>>> C(D_i) @> G{i, i+1} >> C(D_{i+1})[-2]\{2n\} @>>> Co(G_{0,1}) @>>> 0\\
@. @VF_{i,i+1}VV @VF_{i+1, i+2}[-2]\{2n\}VV @VidVV @.\\
0 @>>> C(D_{i+1}) @>G_{i+1,i+2}>> C(D_{i+2})[-2]\{2n\} @>>> Co(G_{0,1})@>>> 0{\rm .}
\end{CD}\]
\end{proof}

\begin{remark}
Although we do not prove it in this paper, we believe that the results of Propositions \ref{usefulseq} and \ref{anotherseq} hold for standard Khovanov homology over the integers, allowing analogues of results such as those of the next section to be deduced in this setting.
\end{remark}

\section{Topological and Structural results}
\label{topsection}

With Proposition \ref{usefulseq} in hand, we can now begin to prove Theorems \ref{changesn}, \ref{snunknot}, and \ref{knotsfromU}.  We note that Propositions \ref{usefulseq} and \ref{anotherseq} seem to contain much of the same information from our point of view, but we suspect that there are some useful applications of Proposition \ref{anotherseq} yet to be uncovered which make use of the fact that $Co(G_{0,1})$ is such a simple complex.

\begin{proof}[Proof of Theorem \ref{changesn}]
Let us work in the equivariant setting.  First note that the commutative diagram in Proposition \ref{usefulseq} can in fact be extended arbitrarily upwards.  This is because for any $l \geq 1$, we can add $l$ negative full twists to $K_0$ forming $\widetilde{K}_0 = K_{-l}$, and then make use of the short exact sequences for $C(\widetilde{K}_j) = C(K_{j-l})$.

Now suppose we are in the situation of Theorem \ref{changesn} where $s_n(K_{-1}) > s_n(K_0)$.

From Proposition \ref{usefulseq} we see that we have the row-exact commutative diagram

\[\begin{CD}
@>>> H^0(K_{-1})\{-2n\} @>>> H^0(K_{0})\{-2n\} @>>> M^2 @>>>\\
@. @VVV @VVV @VidVV\\
@>>> H^{2p}(K_{p-1})\{2np\} @>>> H^{2p}(K_p)\{2np\} @>>> M^2 @>>> {\rm .}
\end{CD}\]

Since the free parts of $H^0(K_{-1})$ and $H^0(K_{0})$ do not lie in the same quantum degrees by hypothesis and $M$ is torsion, this forces the map $H^0(K_{0}) \rightarrow M^2$ to be non-zero.  By commutativity of the righthand square, this also forces $H^{2p}(K_p) \rightarrow M^2$ to be non-zero, and in particular we have $H^{2p}(K_p) \not= 0$.
\end{proof}

We note that with a little more work we could say exactly what quantum degrees of $H^{2p}(K)$ are non-zero, in terms of $s_n(K_{-1})$, $s_n(K_0)$, and $p$.  Such exact information could be useful in investigating whether the $s_n$ homomorphisms are equivalent.  This precise knowledge is not necessary however to deduce Corollary \ref{cheapcor}, which follows immediately.

\begin{proof}[Proof of Theorem \ref{snunknot}]
Let us work in the equivariant setting.  Suppose we have the hypotheses of Theorem \ref{snunknot}.  Let $p \geq 2$, and consider the following part of the commutative diagram of Proposition \ref{usefulseq}

\[\begin{CD}
@>>> H^{-2p-2}(K_0) @>>> H^{-2p-2}(K_1) @>\psi >> M^{2p-2} @>>>  \\
@. @VVV @VVV @VidVV  \\
@>>> H^0(K_{p-1})\{2n(p-1)\} @>>> H^0(K_p) \{2n(p-1)\}  @> \phi >> M^{2p-2} @>>> {\rm .}
\end{CD}\]

Observe that since by hypothesis $K_0$ is the unknot we have $H^{-2p-2}(K_0) = H^{-2p-1}(K_0) = 0$ so that $\psi$ is an isomorphism.  Then the commutativity of the square involving both $\psi$ and $\phi$ tells us that $\phi$ restricted to the torsion part of $H^0(K_p)$ is a surjection.  Therefore there exists a decomposition $H^0(K_p) = Fr \oplus tor$ into free and torsion $\C[a]$-modules such that $\phi \vert_{Fr} = 0$.  But if $s_n(K_p) \not= s_n(K_{p-1})$ then we must have $\phi \vert_{Fr} \not= 0$, hence a contradiction.
\end{proof}

\begin{proof}[Proof of Theorem \ref{knotsfromU}]
Suppose now that we have the hypotheses of Theorem \ref{knotsfromU}.

First of all we would like to see that $M = \Delta$, the torsion part of $H(K_1)$.  This follows directly from the first row of the commutative diagram in Theorem \ref{usefulseq} and the fact that the map $H^0(K_0) \rightarrow H^0(K_1)$ is onto the free part of $H^0(K_1)$.  Indeed, if this map were not, we would either have a non-torsion part of $M^{-1}$ or we would have $s_n(K_1) < 0$.

Now let $p \geq 2$ and consider the following two commutative diagrams with exact rows

\[\begin{CD}
@>>> H^{i - 2(p-1)}(K_1) @>>>  M^{i-2p+2} @>\phi>> H^{i - 2(p-1) + 1}(K_0) @>>> \\
@. @VVV @VidVV @VVV  \\
@>>> H^i(K_{p})\{2n(p-1)\} @>>> M^{i - 2p+2} @>\psi>> H^{i+1}(K_{p-1})\{2n(p-1)\} @>>>{\rm ,}
\end{CD}\]

\[\begin{CD}
@>>> M^{i - 2p + 1} @>\phi'>> H^{i - 2(p-1)}(K_0) @>>> H^{i - 2(p-1)}(K_1) @>>>  \\
@. @VidVV @VVV @VVV  \\
@>>> M^{i - 2p + 1} @>\psi'>> H^i(K_{p-1})\{2n(p-1)\}  @>>>  H^i(K_{p})\{2n(p-1)\}  @>>> {\rm .}
\end{CD}\]

From the first diagram observe that $\phi = 0$ since $H(K_0)$ is non-torsion.  This implies that $\psi = 0$ by commutativity of the rightmost square.  For the same reason in the second diagram we see $\phi' = 0$, which implies that $\psi' = 0$ by the commutativity of the leftmost square.  This means that each row gives rise to short exact sequences

\[ 0 \rightarrow H^i(K_{p-1}) \{ 2n(p - 1) \} \rightarrow H^i(K_p)\{ 2n(p - 1)\} \rightarrow M^{i - 2p + 2} \rightarrow 0 {\rm .} \]

\noindent With this in hand, to prove the theorem it remains to see that every such short exact sequence splits to give isomorphisms

\[ H^i(K_p)\{ 2n(p - 1)\} = H^i(K_{p-1}) \{ 2n(p - 1) \} \oplus M^{i - 2p + 2} {\rm .} \]

A splitting map is found by running anticlockwise around the square

\[\begin{CD}
H^{i - 2(p-1)}(K_1) @>>>  M^{i-2p+2}  \\
 @VVV @VidVV  \\
 H^i(K_{p})\{2n(p-1)\} @>>> M^{i - 2p+2} {\rm ,}
\end{CD}\]

\noindent from $M^{i - 2p+2}$ to $H^i(K_{p})\{2n(p-1)\}$, which is possible since the top row of the square is an isomorphism when restricted to the torsion part of $H^{i - 2(p-1)}(K_1)$.
\end{proof}

\begin{proof}[Proof of Theorem \ref{knotsfromU2}]
We can copy the proof of Theorem \ref{knotsfromU} here.  In fact, this situation is simpler since there is no torsion hence every short exact sequence splits.  The one almost delicate point is to deduce that the map appearing in the top row commutative diagram in Theorem \ref{usefulseq}

\[ F_{0,1}: H^0(U) = H^0(K_0) \rightarrow H^0(K_1) \]

\noindent is an injection.  We know that it is an injection equivariantly and furthermore we have a description of the chain-homotopy type of the equivariant complex given in the discussion following the statement of Theorem \ref{snunknot}.  So it follows that $F_{0,1}$ is an injection in the unreduced case which is obtained by setting $a=0$ in the equivariant chain complexes.  The reduced case then follows from $F_{0,1}$ being an injection in the unreduced case and the generalized universal coefficients theorem for principal ideal domains.
\end{proof}

Earlier we promised a discussion of the case when $s_n(K_1) \not= 0$.  Notice that in this case our argument in the proof of Theorem \ref{knotsfromU} goes through as before for all $H^i(K_p)\{ 2n(p - 1)\}$ except when $2p - i = 2$.  Hence we can determine the homology groups of $K_p$ in terms of $H(K_{p-1})$ and $p$ except for $H^{2p-2}(K_p)$.  To fix the remaining homology group it suffices to know the image of the map $H^0(K_0) \rightarrow H^{2p-2}(K_{p-1})\{ 2n(p - 1) \}$.  We do not give a proof of this fact since it is not needed for our main application of Theorem \ref{knotsfromU}.

To begin our proof of Theorem \ref{noncomplete}, we first collect a few results from the literature on hyperbolic 3-manifolds.  We state the first theorem not as strongly as Thurston proved it, but strongly enough for us to use.

\begin{theorem}[Hyperbolic Dehn Surgery \cite{Thurston}]
\label{HDS}
Let $M$ be a cusped hyperbolic $3$-manifold with a distinguished cusp.  We write $M(1/p)$ for the result of filling the distinguished cusp with filling coefficient $1/p$.  Then $M(1/p)$ is hyperbolic except for a finite set of filling slopes and $M(1/p)$ converges to $M$ in the geometric topology as $p \rightarrow \infty$.
\end{theorem}




\begin{figure}
\centerline{
{
\psfrag{Ui}{$U_i$}
\psfrag{U}{$U$}
\psfrag{Tani}{$Tan_i$}
\psfrag{ol}{$\overline{Tan_i}$}
\psfrag{x2}{$x_2$}
\psfrag{x3}{$x_3$}
\psfrag{x4}{$x_4$}
\psfrag{T+(D)}{$T^+(D)$}
\includegraphics[height=1.2in,width=2.2in]{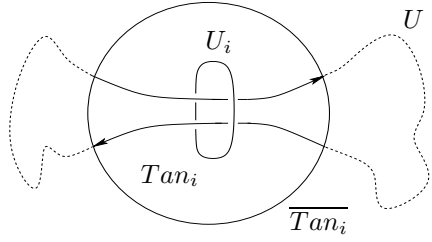}
}}
\caption{This diagram accompanies the statement of Theorem \ref{akio}.  We have drawn a tangle in a small $3$-ball $Tan_i \subset B^3$ which is a subtangle of the link $(U \cup U_i) \subset S^3$.  It consists of all of $U_i$ and two strands of $U$ which intersect a disc bounded by $U_i$ in two points, with signed count $0$.  (The rest of $U$ has been drawn schematically as a dotted line).  We denote by $\overline{Tan_i}$ the complement to this tangle so that $Tan_i \cup_\partial \overline{Tan_i} = U \cup U_i$.}
\label{tangle}
\end{figure}

We shall also need a result of Kawauchi's concerning special knots $K^*$ in $S^3$.

\begin{theorem}[Kawauchi \cite{Kaw}]
\label{akio}
For every $m>1$ there exists an $(m+1)$-component link 

\[U \cup U_1 \cup U_2 \cup \cdots \cup U_m \subset S^3 {\rm ,}\]

\noindent where $U$ is the unknot and $U_1 \cup U_2 \cup \ldots \cup U_m$ is the $m$-component unlink, satisfying the following properties:

\begin{enumerate}
\item Each $U_i$ bounds a disc intersecting $U$ in two points with signed count $0$.
\item For $i \not= j$, the link $U \cup U_i$ is distinct from the link $U \cup U_j$.
\item For any $i$, the result of $+1$-surgery on $U_i$ turns $U$ into a smoothly slice knot $K^*$, which is independent of $i$.
\item Define the tangles $\overline{Tan_i}$ as in Figure \ref{tangle}.  Each tangle $\overline{Tan_i}$ is hyperbolic, as is the branched double cover of each  $\overline{Tan_i}$.
\end{enumerate}
\end{theorem}

\begin{proof}[Proof of Theorem \ref{noncomplete}]
Consider Figure \ref{tangle}, $Tan_i$ is an example of a \emph{simple} tangle (in other words prime and atoroidal).  Furthermore, we know by item (4) of Theorem \ref{akio} that $\overline{Tan_i}$ (the complement of $Tan_i$) is hyperbolic.

We are in the situation where we can apply Lemma 2 of \cite{Soma}.  This tells us that if we glue back $Tan_i$ to $\overline{Tan_i}$, then the result (which is $Tan_i \cup_\partial \overline{Tan_i} = U \cup U_i$) is a hyperbolic link.

We now write $K^i_N$ for the result of doing $(1/N)$-surgery on $U_i$ to the knot $U$, so that $K^i_0 = U$ for each $i$.  By item (3) of Theorem \ref{akio}, we see that $K^i_1$ is the knot $K^*$ for each $i = 1, 2, \ldots, m$.

Since the complement of $U \cup U_i$ is atoroidal for each $i$, Theorem \ref{HDS} tells us that the complement of $K^i_N$ is hyperbolic for large enough $N$ and that these complements converge in the geometric topology to the complement of $U \cup U_i$ as $N \rightarrow \infty$.  Since the meridians of the $K^i_N$ converge to the meridian to $U$, the sequence of knots $K^i_N$ determines the link complement to $U \cup U_i$ as well as the meridional curve to $U$.  By filling along the meridian and taking $U$ isotopic to any longitude relative to the meridian, we see this determines $U$ inside the solid torus complement to $U_i$.  Since there is only one way to fill the boundary of this solid torus to get $U = K^i_0$ unknotted inside $S^3$, we have determined the whole link $U \cup U_i$.

Hence there exists an $N$ such that the complement to $K^i_N$ is not diffeomorphic to the complement to $K^j_N$ whenever $i \not= j$.  Since the knot complement determines the knot, we know that for this $N$ we have $K^i_N \not= K^j_N$ whenever $i \not= j$.  This set $\{K^1_N, K^2_N, \ldots , K^m_N \}$ will be the $m$ distinct knots we are required to exhibit.

Because $K^*$ is slice we have $s_n(K^i_1 = K^*) = 0$ for all $i = 1, 2, \ldots m$.  This means that we can apply Theorem \ref{knotsfromU} to see that $H(K^i_N) = H(K^j_N)$ for all $1 \leq i,j \leq m$.

It remains to see that each $K^i_N$ is prime and not $2$-bridge.  Primeness follows from the hyperbolicity of $K^i_N$.

The branched double cover of $K^i_N$ is a Dehn filling of the branched double cover of $\overline{Tan_i}$, with filling slope determined by $N$.  Again, Theorem \ref{HDS} inplies that for $N$ large enough, the branched double cover of $K^i_N$ is hyperbolic.  We know that branched double covers of $2$-bridge knots are lens spaces, which are not hyperbolic.  Hence $K^i_N$ is not $2$-bridge.
\end{proof}

\begin{figure}
\centerline{
{
\psfrag{Ui}{$U_i$}
\psfrag{Tani}{$Tan_i$}
\psfrag{x1}{$x_1$}
\psfrag{x2}{$x_2$}
\psfrag{x3}{$x_3$}
\psfrag{x4}{$x_4$}
\psfrag{T+(D)}{$T^+(D)$}
\includegraphics[height=3.2in,width=3.2in]{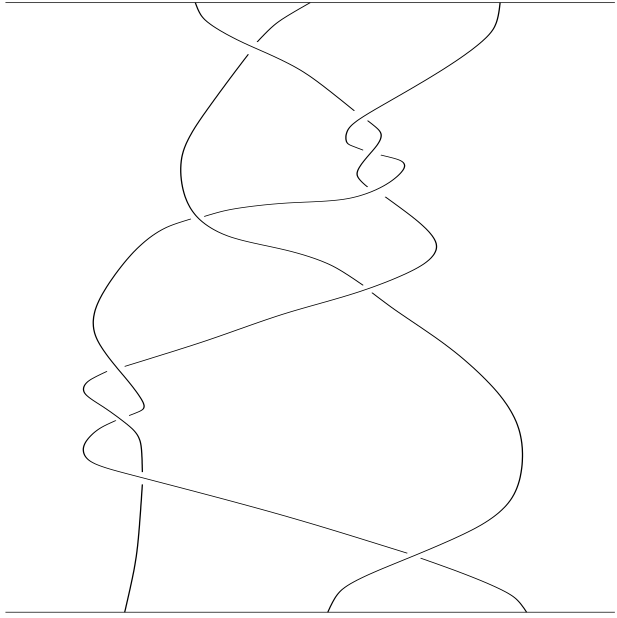}
}}
\caption{Here is an example of a Brunnian pure braid - a pure braid with the property that the removal of any strand results in a trivial braid.}
\label{brunnian}
\end{figure}

\subsection{An example of the construction of a knot pair with isomorphic knot homologies}
\label{examplesandthat}

Kawauchi used the theory of \emph{almost identical imitation} to create knots $K^*$ with multiple unknotting sites \cite{Kaw}.  In the proof of Theorem \ref{noncomplete} we used these knots $K^*$ in an essential way to create distinct knots with isomorphic Khovanov-Rozansky knot homologies.  If we wished to draw a diagram of such knots it would be necessary to understand in detail the theory of almost identical imitation.  However, if one is prepared to work on a more \emph{ad hoc} basis then it is easy to create examples of knots with isomorphic knot homologies.

One such \emph{ad hoc} construction is based on pure Brunnian braids (pure braids that become equivalent to a trivial braid when any strand is removed).  We have drawn an example of such a braid (on three strands) in Figure \ref{brunnian}.

From the braid drawn in Figure \ref{brunnian} we obtain the tangle drawn in Figure \ref{megaexample}.  This tangle can be completed to a knot by filling the slots $X, Y, Z$ with other tangles.  We now abuse notation by referring to the tangle corresponding to the chain complex $T_i$ itself by $T_i$.  We denote by $K^X_i$ the knot obtained by filling $X$ with $T_i$, $Y$ with $T_1$, and $Z$ with $T_{-1}$, and denote by $K^Y_i$ the knot obtained by filling $X$ with $T_1$, $Y$ with $T_i$ and $Z$ with $T_{-1}$.

Note that $K^X_0 = K^Y_0 = U$, the unknot and that $K^X_1 = K^Y_1$.  Furthermore since $K^X_1$ can be transformed into the unknot both by a positive-to-negative crossing change (in place $X$, say) and by a negative-to-positive crossing change (in place $Z$), we must have  $s_n(K^X_1) = s_n(K^Y_1) = 0$.

Hence it follows from Theorem \ref{knotsfromU} that $K^X_i$ and $K^Y_i$ have isomorphic homologies for all $i \geq 2$.

One can check that $K^X_2 \not= K^Y_2$ using \emph{SnapPea}.  In fact, they have different hyperbolic volumes so they are not even mutant by a result of Ruberman's \cite{Rub}.

\begin{figure}
\centerline{
{
\psfrag{X}{$X$}
\psfrag{Y}{$Y$}
\psfrag{Z}{$Z$}
\psfrag{x2}{$x_2$}
\psfrag{x3}{$x_3$}
\psfrag{x4}{$x_4$}
\psfrag{T+(D)}{$T^+(D)$}
\includegraphics[height=3.2in,width=2.9in]{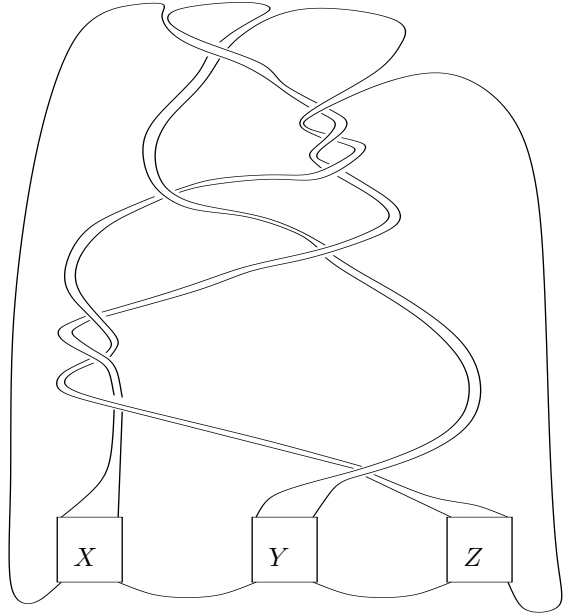}
}}
\caption{Here we show a tangle determined by the braid drawn in Figure \ref{brunnian}.  There are three boundary components to this tangle, each will be filled by some tangle corresponding to the chain complex $T_i$ as in Figure \ref{tk}.}
\label{megaexample}
\end{figure}

\subsection{Pairs of mutant knots with isomorphic knot homologies}
\label{mutantpairs}

The Conway and the Kinoshita-Terasaka (KT) knots are the first (measured by crossing number) example of a pair of mutant knots.  In \cite{MV} Mackaay and Vaz use techniques given by Rasmussen in \cite{Ras2} in order to compute that all reduced Khovanov-Rozansky homologies of the Conway and the KT knots agree.  Since it is easily observed that the KT knot and the Conway knot have unknotting number equal to $1$, we can build upon this computation and give an infinite number of mutant pairs.

\begin{figure}
\centerline{
{
\psfrag{Ui}{$U_i$}
\psfrag{Tani}{$Tan_i$}
\psfrag{Ti}{$T_i$}
\psfrag{x2}{$x_2$}
\psfrag{x3}{$x_3$}
\psfrag{x4}{$x_4$}
\psfrag{T+(D)}{$T^+(D)$}
\includegraphics[height=2.1in,width=2.7in]{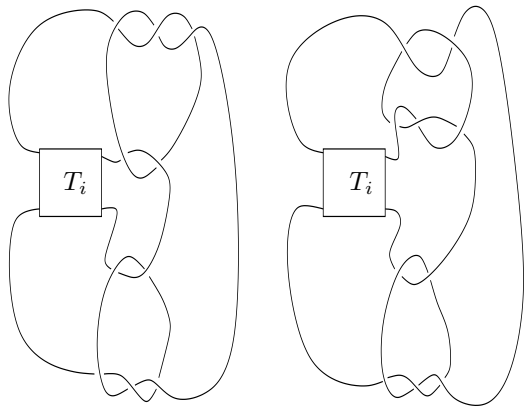}
}}
\caption{This diagram shows two families of knots $K^{KT}_i$ and $K^{C}_i$.  On both sides the unknot $U$ occurs when we put the tangle $T_0$ where indicated $K^{KT}_0 = K^{C}_0 = U$.  When we add the tangle $T_1$ we get the Kinoshita-Terasaka knot $K^{KT}_1$ on the left and the Conway knot $K^{C}_1$ on the right.}
\label{ConwayKT}
\end{figure}

\begin{theorem}
There exist an infinite number of mutant pairs of prime knots that have isomorphic reduced Khovanov-Rozansky homology groups.
\end{theorem}

\begin{proof}
We work with reduced homology.  Consider the two families of knots $K^C_i$ and $K^{KT}_i$ shown in Figure \ref{ConwayKT}.  Since the reduced homologies agree $H(K^C_1) = H(K^{KT}_1)$, $K^C_0 = K^{KT}_0 = U$, and $s_n(K^C_1) = s_n(K^{KT}_1) = 0$, we can apply Theorem \ref{knotsfromU} to see that we have isomorphic homology groups $H(K^C_i) = H(K^{KT}_i)$ for all $i \geq 2$.  Thurston's Theorem \ref{HDS} tells us that $K^C_i \not= K^{KT}_i$ for large enough $i$, and since each is hyperbolic each must be prime.
\end{proof}

\end{document}